\documentclass[reqno,a4paper,12pt]{amsart}
\usepackage{amssymb,enumerate}
	\usepackage{a4wide}

\usepackage[shortlabels]{enumitem}

\newcommand{\calP}{{\mathcal{P}}}
\newcommand{\calT}{{\mathcal{T}}}
\newcommand{\calR}{{\mathcal{R}}}

\newcommand{\Cent}{\ensuremath{{\rm{C}}}}
\newcommand{\NNN}{\ensuremath{{\mathrm{N}}}}

\newcommand{\bB}{{\mathbf B}}

\newcommand{\bG}{{\mathbf G}}

\newcommand{\bN}{{\mathbf N}}
\newcommand{\bS}{{\mathbf S}}
\newcommand{\bT}{{\mathbf T}}
\newcommand{\bm}{{\mathbf m}}

\newcommand{\bZ}{{\mathbf Z}}

\newcommand{\CC}{{\mathbb{C}}}
\newcommand{\FF}{{\mathbb{F}}}
\newcommand{\oF}{\overline{\mathbb{F}}}

\newcommand{\ZZ}{{\mathbb{Z}}}

\newcommand{\tA}{\mathsf A}
\newcommand{\tB}{\mathsf B}
\newcommand{\tC}{\mathsf C}
\newcommand{\tD}{\mathsf D}
\newcommand{\tE}{\mathsf E}

\newcommand{\cC}{{\mathcal C}}

\newcommand{\lp}{{\ell '}}

\newcommand{\cF}{{\mathcal F}}

\newcommand{\cG}{{\mathcal G}}

\newcommand{\cN}{{\mathcal N}}
\newcommand{\cO}{{\mathcal O}}

\newcommand{\cP}{{\mathcal P}}

\newcommand{\Aut}{\operatorname{Aut}}

\newcommand{\Res}{\operatorname{Res}}
\newcommand{\Ind}{\operatorname{Ind}}
\newcommand{\id}{\operatorname{id}}

\newcommand{\Irr}{\operatorname{Irr}}
\newcommand{\Cl}{\operatorname{Cl}}
\newcommand{\Out}{\operatorname{Out}}

\newcommand{\SC}{{\operatorname{sc}}}
\newcommand{\GL}{\operatorname{GL}}

\newcommand{\CSp}{\operatorname{CSp}}
\newcommand{\Sp}{\operatorname{Sp}}
\newcommand{\SO}{\operatorname{SO}}
\newcommand{\Sym}{{\operatorname{S}}}
\newcommand{\Cy}{\operatorname{C}}

\newcommand{\Irrl}{\mathrm{Irr}_{\ell'}}
\def\norm#1#2{{\operatorname N}_{#1}(#2)}
\def\cent#1#2{{\operatorname C}_{#1}(#2)}
\def\ser#1#2{{\mathcal E}(#1 , #2)}
\def\oo#1{\overline{\overline{#1}}}

\newcommand{\w}{\widetilde}
\newcommand{\wh}{\widehat}

\newcommand{\wG}{\ensuremath{{\w G}}}
\newcommand{\wN}{\ensuremath{{\w N}}}

\newcommand{\Z}{\operatorname Z}

\newcommand{\ovF}{\overline \FF}
\newcommand{\ov}{\overline }
\newcommand{\h}{\mathbf h}
\newcommand{\n}{\mathbf n }
\newcommand{\x}{\mathbf x }

\newcommand{\deq}{:=}
\newcommand{\GF}{{{\bG^F}}}
\newcommand{\wGF}{{{{\w\bG}^F}}}

\let\al=\alpha
\let\eps=\epsilon
\let\si=\sigma
\let\la=\lambda

\newcommand{\Gal}{\operatorname{Gal}}
\newcommand{\calG}{{\mathcal G}}

\newtheorem{thm}{Theorem}[section]
\newtheorem{lem}[thm]{Lemma}

\newtheorem{prop}[thm]{Proposition}

\newtheorem{thmA}{Theorem}

\theoremstyle{definition}

\newtheorem{defn}[thm]{Definition}
\newtheorem{notation}[thm]{Notation}

\theoremstyle{remark}
\newtheorem{rem}[thm]{Remark}

\numberwithin{equation}{section} 
\numberwithin{figure}{section} 

\makeatletter
\def\restr #1|#2{\left.#1\right\rceil_{#2}}
\def\spann<#1>{\left\langle#1\right\rangle}

\def\Spann<#1>{\Spann@h#1@}
\def\Spann@h#1|#2@{\left\langle\left.#1\vphantom{#2}\hskip.1em\right.\mid\relax #2 \right\rangle}
\def\Set#1{\Set@h#1@}
\def\Set@h#1|#2@{\left\{\left.#1\vphantom{#2}\hskip.1em\,\right.
\mid\relax #2\right\}}
\def\set#1{\set@h#1@}
\def\set@h#1@{\left\{#1\right\}}

\makeatother

\begin{document}

\title{Inductive McKay condition for finite simple groups of type $\tC$}

\date{August 2016}

\author{Marc Cabanes}
\address{ Institut de Math\'ematiques de Jussieu - Paris Rive Gauche,
Boite 7012,
75205 Paris Cedex 13, France.}
\email{marc.cabanes@imj-prg.fr}
\author{Britta Sp\"ath}
\address{BU Wuppertal, Gau\ss str. 20, 42119 Wuppertal, Germany} 
\email{bspaeth@uni-wuppertal.de}

\thanks{ }



\begin{abstract}
We verify the inductive McKay condition for simple groups of Lie type $\tC$, namely finite projective symplectic groups. 
This contributes to the program of a complete proof of the McKay conjecture for all finite groups via the reduction theorem of Isaacs-Malle-Navarro and the classification of finite simple groups. 
In an important step we use a new counting argument to determine the stabilizers of irreducible characters of a finite symplectic group in its outer automorphism group. 
This is completed by analogous results on characters of normalizers of Sylow $d$-tori in those groups.
\end{abstract}

\maketitle


\section{Introduction}

The McKay conjecture is one of the most important, basic but still open among the counting 
conjectures in the representation theory of finite groups. For $G$ a finite group and $\ell$ a prime, denoting by $\Irr_\lp(G)$ the set of irreducible characters of $G$ of degree prime to $\ell$, the McKay conjecture asserts that $|\Irr_\lp (G)|=|\Irr_\lp (N)|$ where $N$ is the normalizer in $G$ of one of its Sylow $\ell$-subgroups. After the 
reduction theorem of Isaacs-Malle-Navarro \cite{IMN}, showing that this conjecture holds if all simple groups satisfy a stronger version of it, the so-called {\it inductive McKay condition}, several works have recently contributed to the verification of that condition for simple groups. 
After \cite{S12}, \cite{CS13}, \cite{CS15} and \cite{MS15}, it is left open to verify it for odd primes $\ell$ and simple groups of Lie type, whose defining characteristic is some $p\not= \ell$ and whose type is among $\tB$, $\tC$, $\tD$, $^2\tD$, $\tE_6$, $^2\tE_6$, $\tE_7$. 
We solve here type $\tC$. A similar method should apply equally to other types.

As in several other checkings, our starting point is the criterion for the inductive McKay condition introduced in 
\cite[2.12]{S12}. 
The assumptions of this criterion can be divided into a global condition, a local condition and the existence of a bijection. Our
Theorems \ref{MainTh}, \ref{thm_sec6} and \ref{thm_bij} deal separately with each part.

We should mention that Jay Taylor has proved independently Theorem \ref{MainTh} in \cite{T16} with different methods related to generalized Gelfand Graev representations. 

\begin{thmA}\label{thmA}
The inductive McKay condition from \cite[\S 10]{IMN} holds for  projective symplectic groups over finite fields.
\end{thmA}

Our proof starts with the checking of the global condition in Sections \ref{sec_2}-\ref{sec_4}. We propose there
a new method to study the stabilizers of irreducible characters, thus establishing the assumption from Theorem~2.12(v) in \cite{S12}.
In the types $\tB$, $\tC$ and $\tE_7$
we show that this condition holds if and only if two numbers of invariant characters coincide, see Proposition~\ref{prop_stab_counting}. 
Thanks to Brauer's permutation Lemma, one of those numbers can be computed by counting conjugacy classes invariant under certain automorphisms. The other number is also determined in a very elementary way. We provide explicit formulas in the case of a group $G=\bG^F$ where $\bG$ is a symplectic group over $\ov \FF_p$ and $F$ is a Frobenius endomorphism associated with the subfield $\FF_q$ of $\ov \FF_p$.

Then Sections \ref{sec_loc_gen}--\ref{loc_reg} consider the characters of the relevant local subgroups, namely $N=\norm{\bG}{\bS}^F$ where $\bS$ is a  so-called Sylow $d$-torus of $\bG$ for $d$ the order of $q$ mod $\ell$. The groups $\cent{\bG}{\bS}^F$ and $\norm{\bG}{\bS}^F/\cent{\bG}{\bS}^F$ have a well-known ``generic" description (see \cite{BM92}), but knowing enough of the structure of $N=\norm{\bG}{\bS}^F$ to get decisive information on $\Irr (N)$ requires nevertheless a lot of specific work.
The general strategy used here is the one of Section 5 of \cite{CS15}. Most of the additional effort is devoted to the structure of groups closely related to the extended Weyl groups of type $\tC$, using certain elements defined in the associated braid group.

As a corollary of the latter considerations we obtain a bijection for $\w\bG^F$ (a conformal symplectic group) as required in the criterion of \cite[2.12]{S12}, see Theorem \ref{thm_bij}.
In Section \ref{end} we also assemble our previous results to conclude our proof of Theorem~\ref{thmA}.

\medskip

\noindent{\bf Acknowledgement:} The first author thanks the working group ``Algebra und Zahlentheorie" in Wuppertal for its hospitality. 
Both authors thank also the Bernoulli Centre at the EPFL Lausanne for their nice stay.
\medskip


\section{Notation and general considerations} \label{sec_2}
We first explore quite generally how to ensure assumption 2.12(v) of Theorem~2.12 in \cite{S12} for certain simple groups of Lie type. This assumption requires that given a prime $\ell$ every $\chi\in\Irr_{\ell'}(G)$ has a stabilizer in $\Out(G)$ of a specific structure, where $G$ is the universal covering group of the simple group considered. Note that in the case of a simple group of Lie type $\tC$, $\tB$ or $\tE_7$, we have $\Out(G)\cong C_2\times C_m$ ($m\geq 1$) where the first factor corresponds to the so called ``diagonal" automorphisms and the second summand corresponds to field automorphisms. The condition in \cite[2.12(v)]{S12} then clearly reduces to the fact that stabilizers of elements of $\Irr (G)$ in $\Out(G)$ are subproducts of the above direct product.

Let $\bG_{\SC}$ be a simple simply connected algebraic group over $\ov\FF_p$ ($p$ a prime). Let $F_0\colon\bG_{\SC}\to\bG_{\SC}$ a Frobenius endomorphism defining $\bG_\SC$ over $\FF_p$. Let $F_1$ be a power of $F_0$ and $F$ a power of $F_1$, thus inducing an automorphism of $\bG_{\SC}^F$. Denote $G=\bG_{\SC}^F$ and assume it is the universal covering group of the simple group $G/\Z (G)$. We assume fixed $(\bG_\SC ,F)$; only $F_1$ will be allowed to vary.
 
 Let $\bG_\SC\to \w\bG$ be a regular embedding endowed with a Frobenius and field automorphism extending our initial $F$, $F_1$ (see the next section for an explicit choice in the relevant cases). Denote $\w G=\w \bG^F$. Let $\delta\colon G\to G$ be a non-interior diagonal automorphism, that is the conjugation by an element of $\w G\setminus G\cent{\w G}{G}$.

 Recall that for a finite group $H$, $\Cl (H)$ denotes the set of conjugacy classes and $\Irr (H)$ the set of irreducible (ordinary) characters. As above we denote by $\Cl(H)^A$ the $A$-invariant conjugacy classes of $H$, whenever a group $A$ acts on $H$ and $\Cl(H)^a$ for $\Cl(H)^{\spann <a>}$ for any $a\in A$. 
 
 \begin{defn}\label{c} Let
$c(\delta ,F_1)=|\Cl (G)^{\delta F_1}|$ and let $d(\w G ,F_1)$ be the number of $F_1$-stable $\Irr (\w G /G)$-orbits in $\Irr (\w G)$ of length $|\w G/G|$. (Here the action of $\Irr (\w G /G)$ on $\Irr (\w G)$ is given by multiplication with the linear characters.)
 \end{defn}




\noindent

The relevance to our purpose comes from the following

 \begin{prop} \label{prop_stab_counting}
	Assume that the root system of $\bG_{\SC}$ is of type $\tB$, $\tC$ or $\tE_7$, and that $p\not= 2$. Then the following two statements are equivalent: 
	\begin{enumerate}
	\item $c(\delta ,{F_1})=|\w G:G|^{-1}d(\w G ,F_1)$
	for any pair $(\delta ,{F_1})$ as above.
	\item Assumption 2.12.(v) of \cite{S12} holds for $G$, i.e. 
for	every $\chi \in \Irr(G)$ and $(\delta ,{F_1})$, one can have $\chi^{F_1}=\chi^{\delta}$ only if $\chi=\chi^{F_1}=\chi^{\delta}$. 
	\end{enumerate}
 \end{prop}

\begin{proof} We use the same letters to denote the images of $\delta$ and $F_1$ in the outer automorphism group. Note that in the types concerned the outer automorphism group is $\spann<\delta>\times \spann<F_0>$ where $\delta$ (of order 2) induces a non-trivial diagonal outer automorphism, which corresponds to the action of $\w G$ on $G$ .

We also note that $|\w G:G|^{-1}d(\w G ,F_1)= |\Irr (G)^{\spann<\delta ,{F_1}>}| $. 
Indeed since $\w G$ acts on $\Irr (G)$ by $\delta$ and characters of ${\w G}$ restrict to $G$ without multiplicity, the set $\Irr (G)^{\delta }=\Irr (G)^{\w G }$ is the image by restriction of $\Irr'(\w G)$, the set of characters of $\w G$ that restrict irreducibly to $G$. The fibers have cardinality $|\w G:G|^{}$. The restrictions that are ${F_1}$-fixed are the ones that are the restriction of an element of $\Irr(\w G)$ which is sent by ${F_1}$ to a multiple by a linear character of $\w G/G$. This gives our claim.

Now (1) is equivalent to the equality $|\Cl (G)^{\delta {F_1}}|=|\Irr (G)^{\spann<\delta ,{F_1}>}|$. Thanks to
 Brauer's permutation lemma (see \cite[6.32]{Isa}), (1) is then equivalent to 
 \begin{align*} \label{eq}
 	|\Irr (G)^{ \delta {F_1}}|&=|\Irr (G)^{\spann<\delta ,{F_1}>}|.\end{align*}

We now check the equivalence of the above with (2).
 A subgroup $H$ of $\spann<\delta>\times E$ can be different from a product of subgroups of each of those factors only if it projects non-trivially on the first summand and $\delta\notin H$. Let $F_1$ be a generator of the projection of $H$ to $E$. If $H$ has not the required structure it contains $\delta {F_1}$ for some non-trivial $\delta$. If (1) is assumed then
 $\Irr (G)^{\delta^{ } {F_1}}=\Irr (G)^{\spann<{\delta ^{ } , {F_1}}>}$ as seen before. So, if $H$ is now the stabilizer of a character, then this character is $\delta$-invariant. A contradiction showing that (1) implies (2). The converse is trivial. 
\end{proof}

The number $d(\w G ,F_1)$ from Definition~\ref{c} is accessible via the following counting in the dual group. Let $\w\bG^*$ be dual to $\w\bG$ with associated endomorphisms $F_1$ and $F$ in the sense of \cite[8.4]{CE04}. If $s\in\w\bG^*$ we denote by $A(s)$ the group of connected components of the stabilizer of $s \Z ({\w\bG^*})$ in $\w\bG^*$. We abbreviate $\w G^*=\w\bG^*{}^F$.

 \begin{lem}\label{n(G)} Keep the same hypotheses as in Proposition~\ref{prop_stab_counting} above. Then $d(\w G ,F_1)$ is the number of $\w G^*$-conjugacy classes of pairs $( s,\la )$ where $s\in\w G^*$ is semi-simple, $\lambda\in \ser{\cent{\w G^*} s^F}1$ (unipotent characters) is such that the stabilizer of $\la$ under $A(s)^F$ is trivial, and the $\w G^*$-conjugacy class of $(s\Z ({\w\bG^*}),\la )$ is ${F_1}$-stable.
 \end{lem}

\begin{proof} We have seen in the proof of Proposition~\ref{prop_stab_counting} that $d(\w G ,F_1)$ is the number of elements of $\Irr (\w G)$ that restrict irreducibly to $G$ and whose image under $F_1$ is a multiple by a (linear) character of $\w G/G$. 
Through Lusztig's Jordan decomposition of characters of $\w G$, the elements of $\Irr (\w G)$ are parametrized $(s,\lambda )\mapsto 
\chi_{s,\lambda}$ by $\w G^*$-conjugacy classes of pairs $(s,\lambda )$ with $s\in \w\bG^*{}^F_{\rm ss}$, $\lambda\in \ser{\cent{\w G^*} s^F}1$.
The action of $F_1$ is then given by acting on the pair $(s,\lambda )$ (see \cite[3.1]{CS13}). The character $\chi_{s,\lambda}$ restricts irreducibly to $G$ if and only if $A(s)^F_\la =1$ (see \cite[5.1]{L88}). The action of linear characters of $\w G/G$ also corresponds to $\Z ({\w\bG^*})^F$-translation of $s$ (see \cite[8.26]{CE04}). This gives our claim. 
\end{proof}

\section{The global condition in Type $\tC$} \label{sec_4}
In this section we prove Theorem \ref{MainTh} by verifying
the statement of Proposition~\ref{prop_stab_counting}(1) in the particular case where $\bG_{\SC}$ has a root system of type $\tC_l$. 
Accordingly, from now on we assume $\bG_{\SC} =\Sp_{2n}(\ovF_q)$, $\w \bG=\CSp_{2n}(\ovF_q)$, $G=\Sp_{2n}(q)=\GF$ and $\w G:=\CSp_{2n}(q)$.
 
\begin{thm}\label{MainTh} Let $\chi\in\Irr(G)$. 
	For any field automorphism $F_1$ and any diagonal automorphism $\delta$ of $G$ the equality $\chi^{F_1}=\chi^{\delta}$ implies $\chi=\chi^{F_1}=\chi^{\delta}$.

\end{thm} 

Before computing the number of $\delta F_1$-invariant $G$-conjugacy classes in $G$ we describe in general their parametrization and the action of automorphism in terms of their parameters. 
In a second part of this section we compute the relevant number of characters of $\w G$.

\subsection{Conjugacy classes}
We first recall the parametrization of conjugacy classes of $G=\GF =\Sp_{2n}(q)$ (see \cite[2.6]{W63}, \cite[\S 3]{Mi69}). 

In the following we denote by $\cF_q$ the set $\ov\FF_p^\times/\spann<F>$ of orbits of $\ov\FF_p^\times$ under the map $F$ taking the $q$-th power. For any $P\in \cF_q$, we denote by $ P^{-1}\in\cF_q$ the set of inverses of elements of $P$. 

Let $V=(\FF_q)^{2n}$ be the symplectic space over $\FF_q$ whose form $\langle \_ ,\_ \rangle_V$
is given by $\left( \begin{array}{cc} 0&-J_n \\ J_n & 0 \\ \end{array} \right)$ with $J_n=(\delta_{i,n+1-j})_{1\leq i,j\leq n}$. 
We identify $G$ with $\Sp (V)$. 

We recall the parametrization of $\Cl (G)$ following the approach of Cikunov-Milnor (see \cite[\S 3]{Mi69}): 
The conjugacy class containing $x\in G$ corresponds to $(\bm,\Psi_+,\Psi_-)$, where $\bm:\cF_q\times \ZZ_{\geq 1}\to\ZZ_{\geq 0}$ and $\Psi_+,\Psi_-\colon \ZZ_{\geq 1}\to
\FF_q^\times /(\FF_q^\times{})^2$ are defined at the class of $x$ as follows. First $\bm$ is given by the Jordan normal form of $x $ seen as an element of $\GL_{2n}(q)$, namely $\bm(P,j)$ is the number of Jordan blocks of size $j$ associated to any eigenvalue of $x$ belonging to $P$. 

Let $\eps =\pm 1$, then $\Psi_\eps$ is defined as follows. For $j\geq 1$, let $K_{j}(\eps):=\ker( (x-\eps )^j)$ and $V_{2j}'(\eps ):=K_{2j}(\eps)/(K_{2j-1}(\eps)+(x- \eps )K_{2j+1}(\eps))$, an $\bm(\eps ,2j)$-dimensional $\FF_q$-space. Then via $(\bar u,\bar v)_{j,\eps}=
\langle (x-x^{-1})^{2j-1}u,v\rangle_V$ for any $u,v\in(\ker (x-\eps)^{2j})$ one obtains a symmetric non-degenerate form on $V_{2j}'(\eps )$. One defines $\Psi_+$, $\Psi_-$ such that $\Psi_\eps (j)$ gives the Witt type of $(\bar u,\bar v)_{j,\eps}$ and equals $1$ whenever $\bm( \eps,2j)=0$.

\begin{thm}\label{cla} 
Let $\cC$ be the set of triples $(\bm ,\Psi_+,\Psi_-)$ with $\bm:\cF_q\times \ZZ_{\geq 1}\to\ZZ_{\geq 0}$ and $\Psi_+,\Psi_-\colon \ZZ_{\geq 1}\to
\FF_q^\times /(\FF_q^\times{})^2$ that satisfy 
$
\bm( P^{-1},j)=\bm(P,j)$, $\bm( 1,2j-1)\in 2\ZZ$, $\bm(-1,2j-1)\in 2\ZZ$ for every $(P,j)\in \cF_q\times \ZZ_{\geq 1}$, and $ \sum_{(P,j)\in\cF_q\times \ZZ_{\geq 1}}j \cdot \bm(P,j) \cdot |P|=2n$. Then the conjugacy classes of $\GF$ are parametrized by $\cC$ through the description given above. \end{thm} 

Clearly outer automorphisms of $G$ act on $\Cl(G)$ and hence on the parametrizing set $\cC$. The action of diagonal automorphisms has already been (implicitly) used in \cite[6.3.3]{L77} (see also the computation of the centralizers of unipotent elements in \cite{LS12}). We provide a proof for completeness.

\begin{prop}\label{act} 
	Let $x\in G$ and $(\bm ,\Psi_+,\Psi_-)$ the parameter of the conjugacy class containing $x$. 
\begin{enumerate}[(1)]
\item For a field automorphism $F_1$ the element $F_1(x)$ belongs to the conjugacy class $(F_1(\bm),\Psi_+,\Psi_-)$ where $F_1(\bm) (F_1(P),j):= m(P,j)$ for every $(P,j)\in \cF_q\times \ZZ_{\geq 1}$.
 
\item For $\delta$ a non-interior diagonal automorphism of $G$ (induced by an element of $\wG\setminus {\rm Z}( \wG )G$) the conjugacy class containing $\delta(x)$ corresponds to $(\bm,\Psi'_+,\Psi'_-)$, where for every $\eps\in\{\pm 1\}$, $j\geq 1$, one has $\Psi'_\eps(j)\neq \Psi_\eps(j)$ if and only if $\bm(\eps , 2j)$ is odd.
 \end{enumerate}
\end{prop} 
\begin{proof}
%
%

Via the action of $F_1$ on the conjugacy classes of $\GL_n(q)$ one easily understands the induced action on the first parameter. On the other hand it is clear that $F_1$ is induced by an automorphism of the ground field $\FF_q$ of the space $V$ and induces the same on the $V_i'$'s. It preserves the equivalence class of any bilinear form. This completes the proof of (1).

For part (2) we consider the action of the conformal group $\w G$ on $G$. Because of $\w G\leq \GL(V)$, the map $\bm$ is in the triple of $\cC$ corresponding to $\delta(x)$. We now turn to check the effect of $\delta$ on the parameters $\Psi$.

By decomposing $V$ into orthogonal direct sum (see \cite[Th 3.2]{Mi69}) it is sufficient to check the case where $\bm$ has only one non-zero value, i.e. $y:=x-1_V$ has Jordan type $\mathrm{Jor}(d)^{(m)}$ with $md=2n$ and $\mathrm{Jor}(d)$ is the $d\times d$ matrix whose $(i,i+1)$ terms equal $1$ and all others $0$. 

For odd $d$, the above parametrization tells us that the Jordan form gives only one conjugacy class, so the latter is preserved by diagonal automorphisms. 

Assume now $d$ even. Let us recall the associated non-degenerate symmetric form $(\_,\_)_x$ defined on $V'=V/yV$ by $(v+yV,v'+yV)_x:=\langle (x-x^{-1})^{d-1}v,v'\rangle_V$ for any $v,v'\in V$. We must show that the Witt types of $( \_,\_) _x$ and $( \_,\_) _{\delta(x)}$ coincide for even $m$ and always differ for odd $m$.

For every $a\in \GL_n(q)$ and $\la\in\FF_q^\times$, the matrices $D(a):=\left( \begin{array}{cc} a&0 \\ 0 & J_n{}^ta^{-1}J_n \\ \end{array} \right)\in \Sp_{2n}(q)$ and $c_\la := \left( \begin{array}{cc} \lambda I_n&0 \\ 0 & I_n \\ \end{array} \right)$ commute. Moreover the elements $c_\la$ generate $\CSp_{2n}(q)/\Sp_{2n}(q)$. 
When $a$ is a direct sum of unipotent Jordan blocks all of size $d$ a divisor of $n$, then $D(a)-1$ is indeed of type $\mathrm{Jor}(d)^{(2{n\over d})}=\mathrm{Jor}(d)^{(m)}$, and we have just seen that its centralizer covers the quotient $\CSp(V)/\Sp(V)$. So $\CSp(V)$ fixes that conjugacy class. Hence it also fixes the conjugacy class with opposite $\Psi$. This shows our claim in the case where $m$ is even.

In the case where $m=2m'+1$ is odd, let us split the space into a sum $V_1\perp V_2$ where $V_1$ has dimension $2m'd$ and $V_2$ has dimension $d$ (recall that $md$, hence $d$, is even, so this is possible). Both $V_1$ and $V_2$ can be assumed to be endowed with forms of the type above with respect to matrices $J_{m'd}$ and $J_{d/2}$. Let $D(a')\in \Sp(V_1)$ be as in our first case with $a'$ a direct sum of $m'$ unipotent Jordan blocks of size $d$. Let $b$ be a unipotent Jordan block of size $d/2$ and let $E:=\left( \begin{array}{cc} b&bJ_{d/2} \\ 0 & J_{d/2}{}^tb^{-1}J_{d/2} \\ \end{array} \right)$. The latter is clearly in $\Sp (V_2)$, unipotent. The rank of $E-I_d$ is easily checked to be $d$, since the elements on the upper second diagonal of length $d-1$ are $1^{d/2}(-1)^{d/2-1}$ (we use exponential notation for tuples). So $E$ is unipotent of Jordan type $I_d+\mathrm{Jor}(d)$. Now the element $x:=D(a')\perp E\in \Sp (V_1\perp V_2)$ is unipotent of type $I_{2n}+\mathrm{Jor}(d)^{(m)}$. An element of $\CSp (V_1\perp V_2)$ can be formed by taking $c_\la$'s as above with same $\la\in \FF_q^\times$ on $V_1$ and $V_2$. On $V_1$ it commutes with $D(a')$ while on $V_2$ it conjugates $E$ into 
$E_\la=\left( \begin{array}{cc} b&\la bJ_{d/2} \\ 0 & J_{d/2}{}^tb^{-1}J_{d/2} \\ \end{array} \right)$. The bilinear form $(\_,\_)_x$ associated with $x$ is clearly an orthogonal sum of the forms associated with $D(a')$ and $E$. On $V_1$ the form is unaffected by the conjugation, since $D(a')$ is not changed. Let us show that the form coming from $V_2$ (on a line) is multiplied by $\lambda$. Numbering the basis of $V_2$ used for the matrix above as $e_1\dots e_{d}$, the associated orthogonal space can be identified with the line generated by $e_d$ and the form is defined by $(e_d,e_d)_{E_{\la}}=\langle (E_\la-E_\la^{-1})^{d-1}e_d,e_d\rangle_{V_2}$. The elements on the upper diagonal of size $d-1$ in $E_\la-E_\la^{-1}$ are $2^{d/2-1}(2\la)^1(-2)^{d/2-1}$ (exponential notation again), and therefore
$(E_\la-E_\la^{-1})^{d-1}e_d=(-1)^{d/2-1} 2^{d-1}\la e_1=\la (E-E^{-1})^{d-1}e_d$. So it is clear that $(e_d,e_d)_{E_\la}=\la (e_d,e_d)_{E}$, hence our claim.
\end{proof}

Note that by the above statement the analogue of Theorem \ref{MainTh} clearly holds for conjugacy classes. We use the above description to count the number $c_n (\delta ,F_1)$ of conjugacy classes of $G=\Sp_{2n}(q)$ that are $\delta F_1$-invariant where $F_1$ is a field and $\delta$ a non-inner diagonal automorphism of $G$.

 \begin{prop}\label{genfun} Let $q_1=|\FF_q^{F_1}|$. Then 
 	$$ \sum_{n\geq 0}c_n (\delta ,F_1)t^n=\prod_{i\geq 1}{(1+t^{2i})^2\over 1-q_1t^i} .$$
 \end{prop} 

\begin{proof} This is shown via considerations analogous to \cite[p.~38]{W63}.

By Proposition~\ref{act}, we have to count all triples $(m,\Psi_+,\Psi_-)\in \cC$ with $F_1(m)=m$ and such that $m(\pm 1,2j)$ is even for all $j\geq 1$. 
Let $a_n$ be the number of conjugacy classes of $\Sp_{2n}(q)$ without $\pm 1$ as eigenvalue, and 
$f_0(t,q)=\sum a_n t^n$ be its generating function (denoted as $F_0(t^{1/2})$ in \cite[p 37]{W63}). Let $${\mathrm {Part}}_{n}:=\{ (m_1,m_2,\ldots )\mid m_i\in\ZZ_{\geq 0}\text{ and } \sum_{i}im_i=n \},$$
$s(\underline{m}):=|\{ i\mid m_{2i}\neq 0\}|$,
$b_n:=\sum_{\underline m \in {\mathrm {Part}_n}} 2^{s(\underline m)}$ and 
$f_+'(t)=\sum_{n\geq 0} b_n t^n $. Each element $x\in G$ induces an $x$-stable orthogonal decomposition $V=V_0\perp V_\pm$ with $x$ having no eigenvalue $\pm 1$ on $V_0$ and only $\pm 1$ on $V_\pm$. The restrictions are described by the maps $m$ and $(\Psi_+,\Psi_-)$ respectively. The condition that $F_1(m)=m$ means that the retriction of $x$ to $V_0$ has to be in $\Sp (V_0)^{F_1}$. Inputing also the condition on $(\Psi_+,\Psi_-)$ coming from $\delta F_1$-stability, we get $$ \sum_{n\geq 0}c_n (\delta ,F_1)t^n=f_0(t,q_1)f_+'(t)^2.$$

One has $f_0(t,q_1)=\prod_{i\geq 1}{(1-t^i)^2\over 1-q_1t^i}$ according to \cite[p.~37]{W63}. On the other hand by the definition of $b_n$ we easily get 
\begin{align*}
f'_+(t)&=(1+t+t^{2}+\dots )(1+2t^2+2t^{4}+\dots )(1+t^3+t^{6}+\dots )(1+2t^4+2t^{8}+\dots )\dots \\
&= \prod_{i\geq 1}{1+t^{2i}\over 1-t^{i}}. \end{align*} 
So we get indeed $f_0(t,q_1)f_+'(t)^{2}=\prod_{
i\geq 1}{(1+t^{2i})^2\over 1-q_1t^i}$ and our claim.
\end{proof}


\subsection{Counting certain characters} 
In the next step we compute the number $d(\w G,F_1)$ from Definition~\ref{c} using the reformulation of Lemma~\ref{n(G)}.
As above $F_1$ acts on $G$ by raising the matrix entries to the $q_1$-th power, with $q$ a power of $q_1$. Recall $d(\w G,F_1)$ from Definition~\ref{c}.

 \begin{prop}\label{gend} Denote $d'_n:=(q-1)^{-1}d(\w G,F_1)$ for $\w\bG =\CSp_{2n}(\ov\FF_p)$, with $d'_0=1$. Then $$ \sum_{n\geq 0}d'_n t^{n}= \left(\sum_{j\geq 0}t^{j^2+j}\right)\prod_{i\geq 1}{1\over (1-q_1t^{i})(1-t^{2i})}.$$ \end{prop} 

\begin{proof}
With our choice of $\w\bG$, we get for $\w\bG^*$ the connected component of 1 in the Clifford group (see \cite[8.1]{L77}), denoted by $ G^0_{2n+1} $. Let ${\w G^*} =\w\bG^*{}^F$ and $\bZ := \Z ({\w\bG^*})$ a one-dimensional torus. If $s\in\w G^*$ is semi-simple recall $A(s)$ the group of components of the stabilizer of $s\bZ$ in $\w\bG^*$. We now enumerate the ($\w G^*$-conjugacy classes of) pairs $(s,\lambda)$ with the properties given in Lemma~\ref{n(G)} above. 
 
Recall $\w\bG^*/\bZ=(\bG^*)_{\rm ad}=\SO (\ov V)=\SO_{2n+1}(\ov\FF_p)$ is the adjoint group of type $\tB_n$ where $\ov V$ is a $(2n+1)$-dimensional $\ov\FF_q$-vector space. The centralizers in $\SO(\ov V)$ of semi-simple
 elements $s\in \SO_{2n+1}(\FF_q)$ 
 are described for instance in \cite[p.~126]{FS89}. Note that $-1$ has to have even multiplicity as eigenvalue of $s$ since $s$ has determinant 1. Then the eigenspace of the eigenvalue 1 satisfies $\ov V_1\not= 0$ and $2\nmid \dim (\ov V_1)$. The centralizer of $s\in\SO_{2n+1}(\ov\FF_p)$ can be non-connected only if in addition $\ov V_{-1}\not= 0$. Then $A(s)^F=A(s)$ acts by the diagonal outer automorphism on both summands $\SO (\ov V_1)^F$ and $\SO (\ov V_{-1})^F$ of the connected centralizer. All unipotent characters of finite special orthogonal groups are fixed by diagonal automorphisms except the so-called twin characters in type $\tD_{2m}$ (see \cite[5.1]{L88} and the remark following it). This type may occur only as $\SO (\ov V_{-1})$ since $\ov V_1$ has odd dimension. So the condition that $A(s)^F$ has trivial stabilizer on our unipotent character of the connected centralizer is that one of the following two conditions are satisfied 

$(a)$ $A(s)=1$ or equivalently $\ov V_{-1}=0$

$(b)$ $\ov V_{-1}$ has dimension $4m\not= 0$ and the unipotent character of $\SO (\ov V_{-1})^F$ is a twin character (then $|A(s)| =2$).

Unipotent characters are parametrized by classes of symbols $(S,T)$ (see \cite[8.2.(i)]{L77}). Twin characters correspond to the so-called degenerate symbols with $S=T$. The number of classes of degenerate symbols of rank $2m$ is ${\rm p} (m)$, the number of partitions of $m$. So the number of possible unipotent characters is then $2 \cdot {\rm p} (m)$.

Thanks to an easy discussion on eigenvalues and eigenspaces, the condition on $F_1$-stability means that $(s\bZ ,\la )$ is a pair for $\SO_{2n+1} (\FF_{{q_1}})$ where $q_1=|\FF_q^{F_1}|$.


Each class in $ G^*$ from (a) above gives rise to $q-1$ classes in $\w G^*$, while the ones from part (b) give rise to $(q-1)/2$ (see for instance \cite[\S 6.4 page 154]{L77}). So the number we compute is $d'_n= a_n+b_n/2$ where $a_n$ is the number
of $\w G^*$-classes of pairs $(s\bZ ,\la )$ satisfying the condition (a), and $b_n$ is the number
of $\w G^*$-classes of pairs $(s\bZ ,\la )$ satisfying (b). Recall that unipotent characters of $\SO_{2n+1}(q)$ are parametrized by the set of classes of symbols of rank $n$ and odd defect (see \cite[8.2.(i)]{L77}). Denote by $\phi_n$ its cardinality.
Then $$\sum_{n\geq 0} a_n t^{2n+1} = f_0(t^2,q_1)\sum_{m\geq 0} \phi_m t^{2m+1}$$ where as before $f_0(t^2,q):=\prod_{m\geq 1}{(1-t^{2m})^2\over 1-qt^{2m}}$ is the generating function denoted by $F^+_0(t)$ in \cite[p. 37]{W63} and representing the number of conjugacy classes in $\Sp_{2n} (q)$ (or $\SO_{2n} ^+(q)$) without eigenvalue 1 or $-1$. Indeed, by Jordan decomposition in even-dimensional orthogonal groups O$_{2m}(q_1)$, choosing simultaneously a semi-simple element $s'$ without eigenvalue $\pm 1$ and a unipotent character for each factor $\GL_{m(\omega)}(q_1^{|\omega |})$ corresponding with the eigenvalue $\omega$ of $s'$ in the centralizer of $s'$ is combinatorially equivalent to enumerating in the same group the conjugacy classes of elements without $\pm 1$ as eigenvalue, whence the function $f_0(t^2,q_1)$ (see \cite[2.6.12]{W63}) in the formula above.

By the discussion made before about case (b), we also get $$\sum_{n\geq 0} b_n t^{2n+1} = f_0(t^2,q_1)\cdot (\sum_{m\geq 0} \phi_m t^{2m+1})\cdot ( \sum_{m >0}2\cdot {\rm p} (m)t^{4m}).$$
So $$\sum_{n\geq 0}d'_n t^{2n+1}= f_0(t^2,q_1)\cdot(\sum_{m\geq 0} \phi_m t^{2m+1})\cdot\cP (t^4)$$
where $\cP$ is the generating function of partitions, $\cP (t)=\sum_{m\geq 0} {\rm p} (m)t^m=\prod_{m>0} (1-t^m)^{-1}$.

According to \cite[3.4.1]{L77}, one has $\sum_{m\geq 0} \phi_m t^{m} = \cP(t)^2\sum_{j\geq 0}t^{j^2+j}$.
So $$\sum_{n\geq 0}d'_n t^{2n+1}=t f_0(t^2,q_1)\cP (t^2)^2\cP (t^4)\sum_{j\geq 0}t^{2(j^2+j)} .$$ 
We deduce our claim (at $t^2$) by substituting the formula given above for $f_0(t^2,q_1)$.
\end{proof}

\subsection{Proof of Theorem~\ref{MainTh}}
In view of Proposition~\ref{prop_stab_counting} and Lemma~\ref{n(G)}, we now have to compare the generating functions given by Propositions~\ref{genfun} and \ref{gend}. 
After factoring out the term $f_0(t,q_1) = \prod_{i\geq 1}{(1-t^{i})^2\over 1-q_1t^{i}}$, the sought equality amounts to $$\sum_{i\geq 0}t^{i^2+i} =\prod_{i\geq 1} (1-t^{2i})\prod_{i\geq 1} (1+t^{2i})^2 ,$$  an obvious consequence of the Jacobi triple product identity
$$
\sum_{n\in\ZZ} x^{n^2} y^{n} = \prod_{m\geq 1}
\left( 1 - x^{2m}\right)
\left( 1 + x^{2m-1} y\right)
\left( 1 + {x^{2m-1}}{y^{-1}}\right).$$


\section{The local condition - general tools} \label{loc}
\label{sec_loc_gen}
Our goal in this and the following section is to prove Theorem \ref{thm_sec6}, namely that assumption~2.12(vi) of \cite{S12} is satisfied for a choice of $G$, $\w G$, $N$ and $D$ adapted to quasi-simple groups of type $\tC_l$. This assumption requires that the stabilizers of some characters of local subgroups have a particular structure. Like in \cite{MaH0} we choose the local subgroups (called $N$ in \cite[2.12]{S12}) as normalizers of Sylow $d$-tori. Sylow $d$-tori have been introduced in \cite[3.C]{BM92} under the name of Sylow $\Phi_d$-tori. 

In the case where the underlying root system is of type $\tA_l$ the analogous statement has been proven in Theorem 5.1 of \cite{CS15}. In \cite[\S 3]{MS15} this property of characters of normalizers of Sylow $d$-tori has been proven for $d\in\{1,2\}$. (In a more general sense Prop.~3.6 of \cite{S12} can be seen as an analogue for the case where the characteristic of the underlying field coincides with the prime $\ell$, but there the local groups to consider are different.)

In this section we generalize the considerations of \S 5 of \cite{CS15} and clarify which computations have to be done for each type and which arguments from \cite[\S 5]{CS15} can be transferred, see Theorems \ref{thm_loc_gen} and \ref{thm_loc_gen_reg}. We then distinguish two cases, depending on the centralizer of a Sylow $d$-torus being a torus or not. 
We start by introducing the relevant notation and recalling some definitions for later use. In the next section we consider in turn the regular and non-regular case and prove \ref{thm_sec6} by applying Theorems \ref{thm_loc_gen} and \ref{thm_loc_gen_reg}.
 
In contrast to the preceding section where the symplectic group over a finite group was considered as a matrix group we see here this group as a finite subgroup of an algebraic group whose elements are subject to the Chevalley relations. This allows to formulate some results in greater generality. We follow here the notation introduced in \cite[\S 2]{MS15}.

\begin{notation}\label{not_loc_gen}
Let $p$ be a prime. Let $\bG$ be a simply connected simple algebraic group over an algebraic closure $\ov\FF_p$ of $\FF_p$. 
Let $\bB$ be a Borel subgroup of $\bG$ with maximal torus $\bT$. Let $\Phi$, $\Phi^+$ and $\Delta$ denote the set of roots, positive roots and simple roots of $\bG$ that are determined by $\bT$ and $\bB$. Let $\bN:=\norm{\bG}{\bT}$, $W:=\norm{\bG}{\bT}/\bT$ the Weyl group of $\bG$, and let $\rho: \bN \rightarrow W$ be the defining epimorphism. 

For the calculations with elements of $\bG$, we use the Chevalley generators subject to the Steinberg relations as in \cite[Thm.~1.12.1]{GLS3}, i.e., the elements $\x_\al(t_1)$, $\n_\al(t_2)$ and $\h_\al(t_2)$ ($\al\in \Phi$, $t_1,t_2\in \ov\FF_q$ with $t_2\neq 0$) defined therein.

The following endomorphisms of $\bG$ will be of particular interest. Let $F_0:\bG\rightarrow \bG$ the {\it field automorphism of $\bG$} given by 
$$F_0(\x_\al(t))=\x_\al (t^p) \text{ for every } t \in \overline \FF_q \text{ and } \al\in \Phi.$$
Any length-preserving automorphism of $\Phi$ stabilizing $\Delta$ induces a {\it graph automorphism $\gamma$ of $\bG$} satisfying 
$$\gamma(\x_\al(t))=\x_{\gamma(\al)}(t)
\text{ for every } t \in \oF_q \text { and } \al\in\pm\Delta.
$$

Clearly $F_0$ and $\gamma$ commute.
Let $E_0$ be the subgroup $\Gamma_{\overline K}\leq \Aut(\bG)$ from \cite[Def.~1.15.5]{GLS3} corresponding to the length-preserving automorphisms of $\Delta$.  

Let $r$ be the rank of $\Z(\bG)$ (as abelian group) and $\bZ:=(\oF_q^\times)^r$ a torus. Via the identification of $\Z(\bG)$ with a subgroup of $\bZ$ we set
$\w \bG:=\bG\times_{\Z(\bG)} \bZ$. 
We obtain a regular embedding $\iota_{reg}:\bG\rightarrow \w\bG$ in the sense of 
\cite[15.1]{CE04}.
Let $\w \bT:=\bT\bZ$ and $\w \bN:=\bN\bZ$. Like in \cite[2.B]{MS15} we extend the field and graph endomorphisms introduced above to endomorphisms of $\w\bG$. 

We consider the Steinberg endomorphism $F:=F_0^m\gamma$ for some positive integer $m$ with $\gamma \in E_0$.
 Hence $F$ defines an $\FF_q$-structure on $\w\bG$ with $q:=p^m$. We let $G:=\bG^F$ and $\w G:=\w\bG^F$. 

Let $D$ be the subgroup of $\Aut(G)$ generated by $F_0$ and graph automorphisms commuting with $F$. Then $\w G\rtimes D$ is well defined and induces all automorphisms of $G$, see \cite[Thm.~2.5.1]{GLS3}. For later it is relevant that by definition $D$ acts on the set of $F$-stable subgroups of $\bG$. 

Let $V:=\Spann <\n_\al(-1)| \al\in \Phi>$, $H:=\bT \cap V$ and $\phi$ be the automorphism of $\ZZ \Phi$ of finite order induced by $F$. Let $E_1=\cent{E_0}{\gamma}$, $\exp(E_1)$ its exponent, let $e:=o(\phi)\exp(E_1) |V|$, and let $E:=\Cy_{em}\times E_1$ act on $\bG^{F_0^{em}}$ such that a generating element $\wh F_0$ of $\Cy_{em}$ acts by $F_0$ and the second factor in the natural way (by graph automorphisms). We denote by $\wh F_1$ the element $\gamma\wh F_0^m$ of $E$ inducing the automorphism $F$ on $\bG^{F_0^{em}}$.
\end{notation}

\begin{defn}[{\cite[Def.~5.7]{CS15}}]
Let $ Y\lhd X$ be finite groups. We say that {\it maximal extendibility holds with respect to $Y\lhd X$} if every $\chi\in \Irr(Y)$ extends (as irreducible character) to $X_\chi$ (stabilizer). Then, an {\it extension map with respect to $Y\lhd X$} is a map 
$$\Lambda: \Irr(Y)\rightarrow \bigcup_{Y\leq I \leq X}\Irr(I)$$
such that for every $\chi\in\Irr(X)$ the character $\Lambda(\chi)\in \Irr(X_\chi)$ is an extension of $\chi$. 
\end{defn}

When $ Y\lhd X$ and $\chi\in\Irr (Y)$, $\chi '\in\Irr (X)$,  let us recall the notations $\Irr (X\mid \chi)$, $\Irr (Y\mid \chi ')$,  denoting the set of irreducible components of $\Ind^X_Y(\chi)$ and $\Res^X_{Y}(\chi ')$, respectively.

We now formally give a criterion to check assumption 2.12(iv) of \cite{S12}. This summarizes and generalizes considerations from \cite[\S 5]{CS15} that appeared there only in the context of type $\tA$.

\begin{thm}\label{thm_loc_gen}
\label{thm_loc_genii}
Let $d$ be a positive integer. Assume there exists an element $v\in V$ such that for the Sylow $d$-torus $\bS$ of $(\bT,vF)$ the groups $N:=\norm \bG \bS ^{vF}$, $\w N:=\norm {\w\bG} \bS ^{vF}$,  $\wh N:=(\Cent_{\bG^{F_0^{em}}E}(v\wh F))_{\bS}$, $C:=\cent \bG \bS ^{vF}$ and $\w C:=\cent {\w\bG} \bS^{vF}$ satisfy the following conditions: 
\begin{enumerate}[(i)]
\item \label{thm_loc_gen_i}\label{5_3_i}
$\bS$ is a Sylow $d$-torus of $(\bG,vF)$.
\item\label{thm_loc_genii1}
There exists some set $\calT\subseteq \Irr(C)$, such that 
\begin{enumerate}[({ii}.1)]
\item \label{5_3_ii_1}
$\wN_\xi= \w C_\xi N_\xi$ for every $\xi\in \calT$,
\item \label{5_3_ii_2}
$(\wN\wh N )_{\Ind_C^{N}(\xi)} =\wN_{\Ind_C^{N}(\xi)} \wh N_{\Ind_C^{N}(\xi)}$ for every $\xi\in \calT$, and 
\item $\calT$ contains some $\w C$-transversal of $\Irr(C)$.
\end{enumerate}
 
\item \label{thm_loc_genii2}
There exists an extension map $\Lambda$ with respect to $C\lhd N$ such that
\begin{enumerate}[({iii}.1)]
\item \label{5_3_iii_1}
$\Lambda$ is $\wh N$-equivariant. 
\item \label{5_3_iii_2}
If $D$ is non-cyclic, every character $\xi\in\Irr(C)$ has an extension $\wh \xi\in\Irr(\wh N_\xi)$ with $\Res^{\wh N_\xi}_{N_\xi} (\wh \xi)=\Lambda(\xi)$ and $v\wh F\in\ker(\wh\xi)$. 
\end{enumerate}
 
\item \label{thm_loc_genii3}
Let $W_d:=N/C$ and $\wh W_d:=\wh N/C$. For $\xi\in \Irr(C)$ and $\w\xi\in\Irr(\w C_\xi\mid \xi)$ let $W_{\w\xi}:=N_{\w\xi}/C$, $W_{\xi}:=N_{\xi}/C$, $K:=\NNN_{W_d}(W_\xi,W_{\w\xi})$ and $\wh K:=\NNN_{\wh W_d}(W_\xi,W_{\w\xi})$. Then there exists for every $\eta_0\in\Irr(W_{\w\xi})$ some $\eta\in\Irr(W_{\xi}\mid \eta_0)$ such that 
\begin{enumerate}[({iv}.1)]
\item $\eta$ is $\wh K_{\eta_0}$-invariant.
\item \label{5_3_iv_2}
If $D$ is non-cyclic, $\eta$ extends to some $\wh \eta\in \Irr(\wh K_{\eta})$ with $v\wh F\in\ker(\wh\eta)$.
\end{enumerate} 
\end{enumerate}
Then for every Sylow $d$-torus $\bS_0$ of $(\bG,F)$, $N_0:=\NNN_{\bG}(\bS_0)^F$, $\w N_0:=\NNN_\wG(\bS_0)^F$ and $\psi\in\Irr(\w N_0)$ there exists some $\psi_0\in\Irr(N_0\mid \psi)$ such that 
\begin{enumerate}
\item $O_0=(\wGF\cap O_0 )\rtimes (D\cap O_0)$ for $O_0:= \GF (\wGF\rtimes D)_{\bS_0,\psi_0}$, and 
\item $\psi_0$ extends to $(G\rtimes D)_{\bS_0,\psi_0}$.
\end{enumerate}
This ensures assumption 2.12(vi) of \cite{S12} with that choice of groups. 
\end{thm}
\begin{proof}
First we relate the groups $N$ and $\w N$ with the groups $N_0$ and $\w N_0$: the considerations in subsection 5.1 of \cite{CS15}, especially in the proof of Proposition 5.3 establish an isomorphism
$$\epsilon: \w\bG^{F} \rtimes D \rightarrow \cent{\w\bG^{F_0^{em}}E}{v\wh F}/ \spann <v\wh F> .$$ 
Applying this isomorphism proves the equivalence of the following two statements: 
\begin{enumerate}[(i)]
\item Let $\bS_0$ be a Sylow $d$-torus of $(\bG,F)$, $N_0:=\NNN_\bG(\bS_0)^F$, 
$\w N_0:=\NNN_{\w\bG}(\bS_0)^F$ and $\psi\in\Irr(\w N_0)$. 
Then there exists some $\psi_0\in\Irr(N_0\mid \psi)$ such that 
\begin{itemize}
\item $O_0=(\wGF\cap O_0)\rtimes (D\cap O_0)$ for $O_0:= \GF(\wGF\rtimes D)_{\bS_0,\psi_0}$, and
\item $\psi_0$ extends to $(G\rtimes D)_{\bS_0,\psi_0}$.
\end{itemize}
\item For every $\chi\in\Irr(\w N)$ there exists some $\chi_0\in\Irr(N\mid \chi)$ such that 
\begin{itemize}
\item $(\wN \wh N)_{\chi_0}=\wN_{\chi_0} \wh N_{\chi_0}$, and
\item $\chi_0$ has an extension $\w\chi_0$ to $(\wh N)_{\chi_0}$ with $v\wh F\in\ker(\w\chi)$.
\end{itemize}
\end{enumerate}
If $D$ is cyclic, the characters $\chi_0$ and $\psi_0$ have automatically the second property by \cite[Cor.~(11.22)]{Isa}.

We use the extension map $\Lambda$ for $C\lhd N$ from assumption \ref{thm_loc_genii2}. 
Note that for every $t\in(\w\bT^{vF})_\xi$ and $\w\xi\in\Irr(\spann<C,t>\mid \xi)$ the character $\nu\in\Irr(N_\xi)$ given by $\Lambda(\xi)^t=\Lambda(\xi)\nu$ is the lift of a faithful character of $N_\xi/N_{\w\xi}$ by straightforward computations. By the proof of Proposition 5.10 of \cite{CS15} based on Clifford theory we obtain a parametrization of $\Irr(N)$: Let $\calP$ be the set of pairs $(\xi, \eta)$ with $\xi\in\Irr(C)$ and $\eta\in\Irr(W_\xi)$ with $W_\xi=N_\xi/C$. Then 
\begin{align}
\Pi:\calT \rightarrow \Irr(N)\text{ defined by } (\xi,\eta) \longmapsto \Ind^N_{N_\xi}(\Lambda(\xi)\eta)
\end{align}
is surjective and satisfies 
\begin{enumerate}[(i)]
\item $\Pi(\xi,\eta)=\Pi(\xi',\eta')$ if and only if $\xi'=\xi^n$ and $\eta'=\eta^n$ for some $n\in N$,
\item $\Pi(\xi,\eta)^{\si}=\Pi(\xi^{\si},\eta^{\si})$ for every $\si\in \hat N$,
\item $\Pi(\xi,\eta)^{t}=\Pi(\xi^t,\eta \nu_t )$ for every $t\in\w\bT^{vF}$, where $\nu_t\in \Irr(N_{\xi^t})$ is defined by $\Lambda(\xi)^t=\Lambda(\xi^t)\nu_t$. If $\xi^t=\xi$, then $\nu_t$ satisfies $\ker(\nu_t)=N_{\w\xi}$ for any extension $\w\xi$ of $\xi$ to $\spann <C,t>$. 
\end{enumerate} 
For a given $\xi\in\Irr(C)$ and $\w\xi\in\Irr(\w C_\xi\mid \xi)$ the map $t\mapsto\nu_t$ defines a bijection between $\w C_\xi/{C\cent{\w C}G}$ and $\Irr(N_\xi/N_{\w\xi})$.

With the above we prove the required property of characters of $N$. Let $\chi\in\Irr(\w N)$. Then we can choose $\chi_1\in\Irr(N\mid \w\chi)$ such that $\chi_1=\Pi(\xi,\eta)$ for $(\xi,\eta)\in\calP$ and some character $\xi\in\calT$. (Note that $\calT$ contains some $\w C$-transversal of $\Irr(C)$.) We first show that some $C_\xi$-conjugate $\chi_0$ of $\chi_1$ satisfies 
 $$(\wN \wh N)_{ \chi_0}=\wN_{\chi_0} \wh N_{\chi_0}.$$
By the properties of $\calT$ and $\Pi$ every $x\in(\wN \wh N)_{\chi_0}$ can be written as $\w n \wh n$ with $\w n\in \w N_\xi$ and $\wh n\in \wh N_\xi$ after some $N$-multiplication. 
We can write $\w n=tn $ with $t\in \w C_\xi$ and $n\in N_\xi$ by the assumptions on $\calT$. 
Accordingly we can compute as in the proof of Proposition~5.13 of \cite{CS15} the following 
$$\Pi(\xi,\eta)^x=\Pi(\xi,\eta)^{\w n \wh n}=\Pi (\xi , (\eta\nu)^{n \wh n}),$$
where $\nu\in\Irr(W_\xi)$ is the character defined by $\Lambda(\xi)^{t}=\Lambda(\xi)\nu$. Since $\chi_1$ is $\w n \wh n$-invariant this implies $\eta= (\eta\nu)^{ \wh n}$. (Recall $n\in N_\xi$ and hence $\eta^n=\eta$.) Accordingly $\wh n$ stabilizes $\Irr(W_{\w \xi}\mid \eta)$ and hence $\wh n C\in W_\xi \wh K_{\eta_0}$ for some $\eta_0\in \Irr(W_{\w \xi}\mid \eta)$. 
If $\eta$ is the character from \ref{thm_loc_genii3} it is $\wh K_{\eta_0}$-invariant. Then $\eta^{\wh n}=\eta$ and hence $\eta\nu=\eta$. This implies $t\in \w C_{\Lambda(\xi)\eta}$ and 
$$\Pi(\xi,\eta)^{\wh n}=\Pi(\xi,\eta)=\Pi(\xi,\eta)^{\w n}.$$ 
By conjugating $\chi_1$ with elements of $C_\xi$ one obtains some $\chi_0=\Pi(\xi, \eta)$ with $\xi\in \calT$ and $\eta\in\Irr(W_\xi\mid \eta_0)$ with properties as in \ref{thm_loc_genii3}. Then 
$$ (\w N \wh N)_{\chi_0}=\w N_{\chi_0}\wh N_{\chi_0}.$$

Using the character $\wh \xi\in \Irr(\wh N_\xi)$ from assumption (iii.2) and $\Res^{\wh N_\xi}_{\wh N_{\xi,\eta}}(\wh \eta) \in \Irr(\wh N_{\xi,\eta})$ from (iv.2), the construction of the proof of \cite[Prop.~5.13(ii)]{CS15} gives an extension of $\chi_0$ to $\wh N_{\chi_0}/\spann <v\wh F>$. According to the reformulation of the statement in terms of characters of $N$ and $\w N$ from the beginning of the proof, this implies the statement.
\end{proof}

In the study of normalizers of Sylow $d$-tori in \cite{S10a,S10b,CS15} it is important first to understand the situation where the centralizer of the Sylow $d$-torus is a torus, which is equivalent to the fact that $d$ is a regular number of $(\bG,F)$, respectively $W\phi$, see also \cite[Def.~2.4]{S10a}. In this so-called regular case and type $\tA_l$ the considerations in Section 5 of \cite{CS15} have established a strategy to verify the assumptions of Theorem \ref{thm_loc_gen}. The next statement clarifies the necessary computations that have to be performed in that case. 

\begin{thm}\label{thm_loc_gen_reg}
Let $d$ be a regular number for $(\bG,F)$. Assume that there exists an element $v\in V$ with the following properties:
\begin{enumerate}[(i)]
\item \label{thm_loc_geni1}
$\rho(v)\phi$ is a $\zeta$-regular element of $W\phi$ in the sense of \cite{Springer} for some primitive $d$th root of unity $\zeta\in\CC$, see also \cite[Def.~2.5]{S10a}.
\item \label{thm_loc_geni2}
$\rho(V_d)=W_d$ with $V_d:=V^{vF}$ and $W_d:=\Cent_W(\rho(v))$.
\item \label{thm_loc_geni3} \label{5_4_iii}
There exists an extension map $\Lambda_0$ with respect to $H_d\lhd V_d$ such that:
\begin{enumerate}[({iii}.1)]
\item 
$\Lambda_0$ is $\wh V_d$-equivariant with $\wh V_d:=VE\cap \wh N$.
\item \label{5_4_iii_2}
If $D$ is non-cyclic, $\Lambda_0(\la)$ extends for every $\la\in\Irr(H_d)$ to some $\wh \la\in\Irr((\wh V_d)_\la)$ with $v\wh F\in \ker(\wh\la)$.
\end{enumerate}
\item \label{thm_loc_geni4}
Let $C:=\bT^{vF}$, $\w C:=\w\bT^{vF}$, $N:=\bN^{vF}$, $W_d:=N/C$ and $\wh W_d:=\wh N/C$. 
For $\xi\in \Irr(C)$ and $\w\xi\in\Irr(\w C\mid \xi)$ let $W_{\w\xi}:=N_{\w\xi}/C$, $W_{\xi}:=N_{\xi}/C$, $K:=\NNN_{W_d}(W_\xi,W_{\w\xi})$ and $\wh K:=\NNN_{\wh W_d}(W_\xi,W_{\w\xi})$. Then there exists for every $\eta_0\in\Irr(W_{\w\xi})$ some $\eta\in\Irr(W_{\xi}\mid \eta_0)$ such that 
\begin{enumerate}[({iv}.1)]
\item $\eta$ is $\wh K_{\eta_0}$-invariant.
\item \label{5_4_iv_2}
If $D$ is non-cyclic, $\eta$ extends to some $\wh \eta\in \Irr(\wh W_{\eta})$ with $v\wh F\in\ker(\wh\eta)$.
\end{enumerate} \end{enumerate} 
Then for every Sylow $d$-torus $\bS_0$ of $(\bG,F)$, $N_0:=\NNN_{\bG}(\bS_0)^F$, $\w N_0:=\NNN_\wG(\bS_0)^F$ and $\psi\in\Irr(\w N_0)$ there exists some $\psi_0\in\Irr(N_0\mid \psi)$ such that 
\begin{enumerate}
\item 
$O_0=(\wGF\cap O_0 )\rtimes (D\cap O_0)$ 
 for $O_0:= \GF (\wGF\rtimes D)_{\bS_0,\psi_0}$,
\item $\psi_0$ extends to $(G\rtimes D)_{\bS_0,\psi_0}$
\end{enumerate}
This ensures that assumption 2.12(vi) of \cite{S12} is satisfied with that choice of groups. 
\end{thm}

\begin{proof}
We prove that in the above situation the assumptions of Theorem \ref{thm_loc_gen} are satisfied. 

Since $\rho(v)\phi$ is a $\zeta$-regular element of $W\phi$ the Sylow $d$-torus $\bS$ of $(\bT,vF)$ is a Sylow $d$-torus of $(\bG,vF)$ according to \cite[Rem.~3.2 and Lem.~3.3]{S09}. Accordingly \ref{thm_loc_gen}\ref{thm_loc_gen_i} holds. 

Note that $\w C:=\cent{\w \bG}{\bS}^{vF}$ coincides with $\w\bT^{vF}$ and is hence abelian. For $\calT:=\Irr(C)$ every $\xi\in\calT$ satisfies $\w C\subseteq \w N_\xi$ and hence $\wN_\xi=(N\w C)_\xi=\w C N_\xi$. This ensures \ref{thm_loc_gen}\ref{5_3_ii_1}. Analogous computations verify \ref{thm_loc_gen}\ref{5_3_ii_2}. By the choice $\calT:=\Irr(C)$ it contains some $\w C$-transversal. 

The equation $\rho(V_d)=W_d$ and the extension map $\Lambda_0$ allow us to construct an extension map $\Lambda$ as required in \ref{thm_loc_gen}\ref{thm_loc_genii2}: 
Note that $N=V_d C$ and $\wh N=\wh V_d C$. Following the proof of \cite[4.3]{S09} there exists a unique extension map $\Lambda$ with respect to $C\lhd N$ such that 
$$\Res^{N_\xi}_{V_{d,\xi}}( \Lambda(\xi))=\Res^{V_{d,\xi_0}}_{V_{d,\xi}}(\Lambda_0(\xi_0)) \text{ for every }\xi\in\Irr(C)$$
with $\xi_0:=\Res^C_{H_d}(\xi)$.
By the construction $\Lambda$ is $\wh V_d$ and hence $\wh N=\wh V_d C$-equivariant. (This is assumption \ref{thm_loc_gen}\ref{5_3_iii_1}.)
For any $\xi\in\Irr(C)$ and $\la:=\Res^C_{H_d} (\xi)$ we have 
$\wh N_\xi=C (\wh V_d)_\xi$ and $(\wh V_d)_\xi\leq (\wh V_d)_\la$; this implies that there exists an extension $\wh \xi$ of $\xi$ to $\wh N_\xi$ with 
$\Res^{\wh N_\xi}_ {(\wh V_d)_\xi} (\wh \xi)= \Res^{(\wh V_d)_{\la}} _{(\wh V_d)_\xi} (\wh \la) $. This character $\wh \xi$ has then the properties required in \ref{thm_loc_gen}\ref{5_3_iii_2}.
Since the last assumptions coincide one sees that all assumptions of Theorem \ref{thm_loc_gen} hold and hence the statement follows from it.
\end{proof}

The following remark helps in the verification of the assumptions. 

\begin{rem}\label{rem_simpl}
Note that if $F=F_0^m$ and $|E_1|=1$ we have $\wh N=N\rtimes E$, where $E$ acts trivially on $V$ and hence $V_d$. This implies that assumption \ref{thm_loc_gen_reg}\ref{thm_loc_geni3} holds whenever maximal extendibility holds for $H_d\lhd V_d$ since then a $V_d$-equivariant extension map for $H_d \lhd V_d$ exists. 
\end{rem}

\section{The local condition in Type $\tC$}\label{loc_reg}
From now on we concentrate on the case where the underlying root system is of type $\tC_l$ ($l\geq 1$). The aim of this section is to prove the following statement. 

\begin{thm}\label{thm_sec6}
Assume that in the notation of \ref{not_loc_gen} the root system of $\bG$ is of type $\tC_l$ and $ p\not= 2$. Let $d$ be a positive integer, $\bS_0$ be a Sylow $d$-torus of $(\bG,F)$, $N_0:=\NNN_{\bG}(\bS_0)^F$, $\w N_0:=\NNN_{\w\bG}(\bS_0)^F$ and $\psi\in\Irr(\w N_0)$. There exists some $\psi_0\in\Irr(N_0\mid \psi)$ such that 
\begin{enumerate}[(i)]
\item $O_0=(\wGF\cap O_0)\rtimes (D\cap O_0)$ for $O_0:= \GF(\wGF\rtimes D)_{\bS_0,\psi_0}$ and 
\item \label{thm_sec6ii}
$\psi_0$ extends to $(\GF\rtimes D)_{\bS_0,\psi_0}$.
\end{enumerate}
\end{thm}
We must now introduce some specific roots and subgroups. 

\begin{notation}\label{n_typC}
We assume that $p\not=2$, the root system of $\bG$ and its system of simple roots $\Delta:=\{\al_1,\ldots,\al_l\}$ is given as in \cite[1.8.8]{GLS3}. We identify the Weyl group $W$ of $\bG$, a Coxeter group of type $\tC_l$ with the subgroup of bijections $\si$ on $\{ \pm 1, \ldots \pm l\} $ with $\si(-i)=-\si(i)$ for all $1\leq i\leq l$, see also \cite[\S 5]{S10a}. This group is denoted by $\Sym_{\pm l}$. Via the natural identification of $\Sym_l$ with a quotient of $\Sym_{\pm l}$, the map $\rho:\bN\rightarrow W$ induces an epimorphism $\ov\rho\colon V\to \Sym_l$. Note that for this type the group $E_1$ of graph automorphisms is trivial and $D$ is cyclic. 
\end{notation}
We first concentrate on the case where $d$ is a regular number of $(\bG,F)$ in the sense of \cite[Def.~2.4]{S10a}. By \cite[Table 1]{S10a} this implies $d\mid 2l$. For that case we successively ensure the assumptions of Theorem \ref{thm_loc_gen_reg}. At the end of this section we give the proof of Theorem \ref{thm_sec6}.

\subsection{Construction of some Sylow $d$-torus}
For the construction of a Sylow $d$-torus we first choose a Sylow $d$-twist in the sense of \cite[Def.~3.1]{S09} using $v_0:=v_1\dots v_{l}$, where $v_k:=\n_{\al_k}(1)\in \bG$. Note that 
$\rho(v_0)=(1,2,\ldots,l,-1,-2,\ldots, -l)$.
Let $v=v_0^{\frac{2l}{d}}$ and let $d_0$ be the length of the $\ov\rho (v)$-orbits. 
Then 
$$d_0:=\begin{cases} d& 2\nmid d,\\
\frac d 2& 2\mid d.\\
\end{cases}$$
Let $a$ be the number of $\ov\rho(v)$-orbits of $\ov\rho(v)$ on $\{1,\ldots, l\}$, hence $a=\frac{l}{d_0}$.

For the study of $v$ it is helpful to use the braid group $\mathbb B$ of type $\tC_l$, the element ${\bf w}_0 \in \mathbb B$, corresponding to the longest element of $W$ and the natural epimorphism from $\mathbb B$ to $V$ as given in Remark 3.2 of \cite{S09}. By the definition of $v_0$ as Coxeter-like element we see that $v$ is the image of a $d$-th good root of ${\bf w}_0^2$ in $\mathbb B$ in the sense of \cite[Def.~3.8]{BrMi}. By \cite[Rem.~3.3(c)]{S10a} this implies
\begin{align}\label{Bessis_eq}
\rho(V_d)&=\cent W{\rho(v)},
\end{align}
or equivalently $N=V_d\bT^{vF}$ for $V_d:=\cent Vv$. Note $\cent W{\rho(v)}=\Cy_{2d_0}\wr\Sym_{a}$ by computations in $\Sym_{\pm l}$. According to \cite[Lem.~3.3]{S09}, some Sylow $d$-torus $\bS$ of $(\bG,vF)$ is contained in $\bT$. Altogether we see that the assumptions \ref{thm_loc_gen_reg}\ref{thm_loc_geni1} and \ref{thm_loc_gen_reg}\ref{thm_loc_geni2} hold.

\medskip
Next we verify assumption \ref{thm_loc_gen_reg}\ref{thm_loc_geni3} 
for the element $v$. Recall $H_d= \cent Hv$.

\begin{thm}\label{vgood} 
Maximal extendibility holds with respect to $H_d\lhd V_d$. 
\end{thm} 
Following Remark \ref{rem_simpl} this implies assumption \ref{thm_loc_gen_reg}\ref{thm_loc_geni3} in our case. 
\begin{proof}
The proof of the statement is divided into two parts. First the structure of $V_d$ is analysed in detail and then the extensions are constructed. 

Let $\oo v:=\ov\rho (v)$ and $\cO_k$ be the $\oo v$-orbit on $\{1,\dots ,l\}$ containing $k$. Since $a={ l\over d_0}$ the sets $\cO_1 ,\dots ,\cO_{a} $ give a disjoint partition of $\set{1,\ldots,l}$. For $1\leq k\leq {a}$ let $h_k:=\prod_{i\in \cO_k}\h_{2e_i}(-1)$. By Chevalley relations $h_k\in H_d$ is an involution. As for any $h\in H$ there exists $t_i\in \{\pm 1\}$ such that $h=\prod_{i=1}^l \h_{2e_i}(t_i)$ we see 
$H_d=\cent Hv=\spann<h_1> \times \cdots \times\spann<h_{a}>$. 
Hence $H_d$ is an elementary abelian $2$-group of rank $a$.

Similarly we determine elements of $V_d$. Note that $\rho(V_d)$ is a wreath product. Associated to its structure we choose elements in $V_d$ that get mapped to generators of the base groups in $(\Cy_{2d_0})^a$, and involutions generating the symmetric group. If $2\mid d$ let 
$$\ov c_1:=(1,a+1,\ldots, a (d_0-1)+1, -1, -a-1,\ldots, -a (d_0-1)-1 )\in \cent W {\rho(v)}.$$ 
This is the cycle of $\rho(v)$ (written as product of disjoint cycles) containing $1$. If $2\nmid d$ let 
\begin{align*}
\overline c_1':=&(1,2a+1, 4a+1, \ldots, (d-1)a+1, -a-1,\ldots, -(d-1)a-1 )\\&
(-1,-2a-1, -4a-1, \ldots, -(d-1)a-1, +a+1,\ldots, (d-1)a+1 )
 \in \cent W {\rho(v)}
\end{align*}
Let $\overline c_1:=\overline c_1'\prod_{i=0}^{d-1 } (ia+1,-ia-1) \in \cent W {\rho(v)}$. Note that $\overline c_1'$ is the pair of cycles of $\rho(v)$ (written as product of disjoint cycles) containing $\pm 1$. 

Equation \eqref{Bessis_eq} implies that there is some $c_1\in V_d$ with $\rho(c_1)=\ov c_1$. Next we prove that for $\Phi_1:=\Phi\cap {\spann<\pm e_i\mid i\in \cO_1> }$ and $V_{\cO_1}:=\spann< \n_\al (\pm 1)\mid \al\in \Phi_1>$ some $c_1\in V_{\cO_1}\cap V_d$ satisfies $\rho(c_1)=\ov c_1$. By definition $\n_\al(\pm 1)^2\in H$. Hence $c_1$ is contained in $\bT V_{\cO_1}$. Clearly $V_{\cO_1}\cap\bT\leq H$. Hence $c_1\in V$ implies $c_1\in H V_{\cO_1}$. By the Chevalley relations we see that $ H=(V_{\cO_1}\cap H)\times H'$ for some subgroup $H'\leq H$ and hence 
$$H V_{\cO_1} =V_{\cO_1} \times H'.$$ 
Then $c_1=c'_1h$ for unique $c'_1\in V_{\cO_1}$ and $h\in H'$. Here $(V_{\cO_1})^v=V_{\cO_1}$ and $(H')^v=H'$ imply $c'_1\in\cent{V_{\cO_1}}v$. Accordingly some $c_1\in ({V_{\cO_1}})^{vF}$ satisfies $\rho(c_1)=\ov c_1$. By this construction $\spann <c_1, h_1>=(V_{\cO_1})^{vF}$.

For $1\leq k \leq (a-1)$, let $p_k:=\prod_{i=0}^{d_0-1}{v_k}^{v^i}$. Using $\n_\al (\pm 1)^{\n_\beta (\pm 1)}=\n_{s_\beta (\al)}(\pm 1)$ from \cite[1.12.1(j)]{GLS3} together with the commutator formula from \cite[1.12.1(b)]{GLS3} we see 
$$\rho(p_k)(\cO_k)=\cO_{k+1} \text{ and }\rho(p_k)(\cO_{k+1})=\cO_{k}.$$ 
Moreover $p_k\in V_d$ since $v^{d_0}\in \Z(V)$. Like before straightforward computations show $p_k \in V_{\cO_k\cup\cO_{k+1}}$, where $V_{\cO_k\cup\cO_{k+1}}:=\spann< \n_\al (\pm 1)\mid \al\in \Phi_{k,k+1}>$ with $\Phi_{k,k+1}:=\Phi\cap {\spann<\pm e_i\mid i\in \cO_k\cup \cO_{k+1}> }$.

Recall $\oo v:=\overline \rho(v)$. The square of $p_k$ satisfies 
$$p_k^2=\prod_{i=0}^{d_0-1}({v_k}^{v^i})^2=
\prod_{i=0}^{d_0-1} \h_{e_{\oo{v}^i(k)}-e_{\oo{v}^i(k+1)}} (-1)
=\prod_{j\in \cO_k\cup \cO_{k+1}}\h_{2e_j}(-1)=
h_k h_{k+1}\in H_d .$$ 
Note that since the elements $v_1,\ldots, v_{a-1}$ satisfy the braid relations this applies also to $p_k$. For $2\leq k \leq a$ let $c_k:=(c_1)^{p_1 \cdots p_{k-1}}$. 

Clearly it is sufficient to prove that the characters of some $V_d$-transversal in $\Irr(H_d)$ extend to their inertia group in $V_d$.
Recall that $H_d=\spann<h_1>\times \cdots \times \spann <h_a>$. Accordingly every $\la\in \Irr(H_d)$ can be written as $\la_1\times \cdots \times \la_a$ with $\la_k\in\Irr(\spann <h_k>)$. 
Since $V_d$ acts on $H_d$ via permutation of the groups $\spann<h_k>$ and $|\spann<h_k>|=2$ every character of $H_d$ is $V_d$-conjugate to some $\la=\la_1\times \cdots \times \la_a\in\Irr(H_d)$ where there exists some integer $k_0$ ($0\leq k_0\leq a$) with
$$o(\la_k)=\begin{cases}
1 & k\leq k_0,\\
2 & k> k_0.
\end{cases}$$
Note that such a character satisfies 
$$\la^{c_k}=\la \text{ for all } 1\leq k\leq a \text{ and }
\la^{p_k}=\la\text{ for all }1\leq k<a\text{ with }k\neq k_0 .$$
This proves $(V_d)_\la= CS$, where $C:=H_d\Spann <c_{i'}| 1 \leq i'\leq a>$ and $S:=\Spann <p_k| 1 \leq k< a, k\neq k_0 >$.

Let $\Lambda_1$ be an extension map for $\spann <h_1> \lhd \spann <c_1, h_1>$. It exists since $\spann <c_1, h_1>/\spann <h_1>$ is cyclic. Let $\Lambda_k$ be the extension map with respect to $\spann <h_k>\lhd \spann <h_k,c_k>$ that is obtained by conjugating $\Lambda_1$ with $p_1\cdots p_{k-1}$. 

The character 
$$\w \lambda:=\Lambda_1(\la_1)\times\cdots \times \Lambda (\la_a)\in 
\Irr(\spann<h_1,c_1> \times \cdots \times \spann <h_a,c_a>)$$ 
is $p_k$-invariant for $ 1 \leq k< a \text{ with } k\neq k_ 0 $ by the definition of the extension maps $\Lambda_{k'}$. Note that 
$S$ is a supplement of $C$ in $(V_d)_\la$. Since $\rho(S)$ is the direct product of two symmetric groups 
$$ S\cap H_d= \Spann <p_k^2|1\leq k \leq a, \, k\neq k_0> $$ 
by the presentation of the symmetric groups as Coxeter groups. By the above computations of $p_k^2$ we see $ S\cap H_d \leq \ker(\lambda)$. According to \cite[Lem.~5.8(a)]{CS15} $\w\la$ and hence $\la$ extend to some $\wh \la\in\Irr((V_d)_\la)$ with $\wh\la(S)=1$. 
%
\end{proof}

In the next step we verify assumption \ref{thm_loc_gen_reg}\ref{thm_loc_geni4}. Note that $D$ is cyclic. We recall the action of $N:=\bN^{vF}$ on $C:=\bT^{vF}$ from \cite{S10a}, then we describe the groups $W_\xi$ and $W_{\w\xi}$ for $\xi\in\Irr(C)$ and $\w\xi\in \Irr(\w C\mid \xi)$ and study their characters.

\begin{notation}
In the following let $\epsilon=(-1)^{\frac d {d_0}}$. We identify $\Cy_{q^{d_0} - \epsilon}$, the cyclic group of order $q^{d_0} - \epsilon$ with (a subgroup of) $\FF_{q^d}^\times$ and let the group $\Gal(\FF_{q^d}/\FF_q)\times \Cy_2$ act on it by Frobenius endomorphisms and inversion.
Let $\mathcal G\leq \Gal(\FF_{q^d}/\FF_q)\times \Cy_2$ be defined by
$$\mathcal G:= \begin{cases}
\Gal(\FF_{q^d}/\FF_q)\times \Cy_2& \text{, if }d=d_0 ,\\
\Gal(\FF_{q^d}/\FF_q)& \text{, if } d=2d_0.
\end{cases}$$
Note that $\mathcal G$ is cyclic, $2\mid|\mathcal G|$ and the involution of $\calG$ acts on $\Cy_{q^{d_0} - \epsilon}$ by inversion. Accordingly for $\xi\in\Irr(\Cy_{q^{d_0} - \epsilon})$ the statements  $2\mid |\calG_\xi|$ and $o(\xi)\in \{1,2\}$ are equivalent.
 
Let $f$ and $n$ be two positive integers with $2\mid f$. We write $(h_1,\ldots, h_n)\sigma$ for the elements of $\Cy_{f}\wr \Sym_n$, where $h_i\in \Cy_f$ and $\si\in \Sym_n$. Via $\Cy_f\leq \CC^\times$ we define by 
$$(h_1,\ldots, h_n)\sigma \mapsto \left (\prod_i h_i\right )^{\frac{f}{2}}$$ 
a linear character $\mu\in\Irr(\Cy_{f}\wr \Sym_n)$ of order $2$. 
\end{notation}

\begin{prop}\label{prop6_7}

	\begin{enumerate}[(a)]
		\item \label{prop6_7a}
		  $C=\bT^{vF}\cong (\Cy_{q^{d_0}-\epsilon})^a$ is the direct product of $a$ cyclic groups of order $q^{d_0}-\epsilon$. 
		\item \label{prop6_7b}
		Let $\calR$ be a $\calG$-transversal of $\Irr(\Cy_{q^{d_0}-\epsilon})$ and $\xi \in\Irr(C)$ with $\xi=\xi_1\times \cdots\times \xi_a$ for $\xi_i\in\calR$. Then 
		$W_\xi= \bigoplus_{\zeta\in\calR} \calG_\zeta \wr \Sym_{I_\zeta}$,
		where $I_\zeta:=\{i\mid \xi_i=\zeta\}$ for $\zeta\in\calR$.
		\item \label{prop6_7c}
		Let $\w\xi \in\Irr(\w C\mid \la )$, $\zeta_2\in\Irr(\Cy_{q^{d_0}-\epsilon})$ with $o(\zeta_2)=2$, $\mu\in \Irr(\calG_{\zeta_2}\wr \Sym_{I_{\zeta_2}})$ be defined as above and $\nu\in \Irr(W_\la) $ with
		$$\Res^{W_\la}_{{\calG_\zeta \wr \Sym_{n_\zeta}}} (\nu)=\begin{cases}
		1_{\calG_\zeta \wr \Sym_{n_\zeta}} & \zeta\neq \zeta_2,\\
		\mu& \zeta=\zeta_2.\end{cases}$$	
		Then $W_{\w\xi}=\ker(\nu)$ and 
		$ W_{\w\xi} =		\ker(\mu)\times 
		\bigoplus_{\zeta\in\calR\setminus \{\zeta_2\}} \calG_\zeta \wr \Sym_{I_\zeta}$.
				
		\item \label{prop6_7d} 
		For $I\subseteq \{1,\ldots,a \}$ let 
		$$\Delta^{I} \calG:=\{ (h_1,\ldots, h_{a})\id_{\Sym_{a}} \mid h_i = 1_\calG \text{ for }i\notin I
		\text{ and } h_i=h_j \text{ for }i,j\in I\}.$$ 
		Let 	$Z:= (\prod_{\zeta\in\calR} \Delta ^{I_\zeta} \calG )$ and $S$ be the subgroup of elements $\sigma\in \Sym_{a}$ that satisfy:
		\begin{itemize}
		\item for every $\zeta\in\calR$ there exists some $\zeta'\in\calR$ with
		 $\sigma(I_\zeta)= I_{\zeta'}$ and $\calG_\zeta=\calG_{\zeta'}$,
		\item $\sigma(i)< \sigma(j)$ for every $\zeta\in\calR$ and $i,j \in I_\zeta$ with $i<j$.
		\end{itemize}
		Then $\NNN_{W_d}(W_\xi )= (Z W_\xi ) \rtimes S$.
		\item \label{prop6_7e}
		We have $S_\nu=\{\sigma\in S \mid \si(I_{\zeta_2})=I_{\zeta_2}\}$ and $\NNN_{W_d}(W_\xi,W_{\w\xi})=(Z W_\xi) \rtimes S_{\nu}$.
 	\end{enumerate}
\end{prop}
\begin{proof}
Part \ref{prop6_7a} follows from the considerations in \cite[Sect.~5 and 6]{S10a}: $C\cong \mathbf T_1^{vF}\times\cdots \times \mathbf T_a^{vF}$ in the terminology of \cite{S10a}. Hence $C$ is the direct product of $a$ groups isomorphic to $\Cy_{q^{d_0}-\epsilon}$. 

For the proof of \ref{prop6_7b} we write $\xi \in\Irr(C)$ as $\xi_1\times \cdots \times \xi_{a}$ with $\xi_i\in\Irr(\Cy_{q^{d_0}-\epsilon})$. Computations with the Chevalley relations show that $V_d$ acts on $C$ by $\calG$ on each factor and permutation of the $a$ factors. More precisely an element $c_i$ acts by an element generating $\calG$ on the $i$th factor and $p_k$ swaps the $k$th and the $(k+1)$st factor. Each character of $C$ is $V_d$-conjugate to one considered here. Then straightforward computations show the statement. 

In order to prove part \ref{prop6_7c} recall that the center of $\bG$ is generated by the element $\prod_{i=1}^l\h_{2e_i}(-1)$, see also \cite[Table 1.12.6]{GLS3}. This implies that $\widetilde C$ can be embedded into the central product of $(\Cy_{2(q^{d_0}-\eps)})^a$ and $\Cy_{2q-2}$. 
On these groups the induced action of $W_d$ is given by the action of a generating element of $\calG$ on 
$\Cy_{2(q^{d_0}-\eps)}$ via 
$\psi\mapsto \psi^{-\eps((q^{d_0}-\eps)+q)}$. For each $\zeta\in\calR$ one chooses some extension $\widetilde \zeta\in\Irr(\Cy_{2(q^{d_0}-\eps)})$. For each $\zeta\in \calR$ we have $|\calG_\zeta:\calG_{\w\zeta} |\in \{1,2\}$. By the above this implies $\calG_\zeta=\calG_{\w\zeta}$ unless $\zeta=\zeta_2$. 

For part \ref{prop6_7d} we see $\calG^{a}$ as a subgroup of $W_d$. Then $\norm {\calG^{a}} {W_\xi}$ is contained in $Z W_\xi $. On the other hand $S=\norm {\Sym_a}{W_\xi}$ and $\norm {W_d} {W_\xi} \calG^a\leq \calG^a\rtimes S$. Those considerations then prove the statement. 

For the last part \ref{prop6_7e} we see that $Z$ acts trivially on $W_\xi$, and hence $\nu$ is $Z$-stable. Accordingly $Z$ stabilizes $\ker(\nu)$. On the other hand $\norm {S} {W_{\w\xi}}= \norm S {\ker (\nu)}$. For $\sigma\in S$ we have $\nu^\sigma=\nu$ if and only if $\si(I_{\zeta_2})=I_{\zeta_2}$.
\end{proof}

The character theory for the groups $W_\xi$ and $W_{\w\xi }$ is as follows. 
\begin{prop}
For $K:=\NNN_{W_d}(W_\xi ,W_{\w\xi })$ and $\eta_0\in\Irr(W_{\w\xi })$ every $\eta\in\Irr(W_\xi \mid \eta_0)$ is $K_{\eta_0}$-invariant.
\end{prop}
\begin{proof}
As seen in \ref{prop6_7}\ref{prop6_7e} we have 
$K= Z W_\xi \rtimes S_\nu$.
Since $Z$ acts trivially on $W_\xi $ it is sufficient to check if 
there exists some $\eta$ that is invariant under $S_{\nu, \eta_0}$. 

As seen in \ref{prop6_7}\ref{prop6_7d} the groups $W_{\w\xi }$ and $W_\xi $ are direct products of groups indexed by $\calR$ which differ only in the factor associated with $\zeta_2$: in $W_\xi $ the factor associated with $\zeta_2$ is $\calG_{\zeta_2}\wr \Sym_{I_{\zeta_2}}$ and in $W_{\w\xi }$ it is $\ker(\mu)$.

Hence $\eta_0$ is the product of characters $\eta_{0,\zeta}$ ($\zeta\in\calR$) with $\eta_{0,\zeta_2}\in\Irr(\ker(\mu))$ and $\eta_{0,\zeta}\in\Irr(\calG_\zeta\wr\Sym_{I_\zeta})$ for $\zeta\in\calR\setminus \{\zeta_2\}$. Analogously $\eta\in\Irr(W_\xi \mid \eta_0)$ is the product of
$\eta_{\zeta}\in\Irr(\calG_\zeta\wr\Sym_{I_\zeta})$ with $\eta_{0,\zeta}=\eta_{\zeta}$ for $\zeta\neq \zeta_2$ and $\eta_{\zeta_2}\in\Irr(\calG_{\zeta_2}\wr\Sym_{I_{\zeta_2}} \mid \eta_{0, \zeta_2})$. 

Let $\si\in S_{\nu,\eta_0}$. We see that by definition $\si$ acts trivially on the factor related to $\zeta_2$. On the other hand this is the only factor 
in which $\eta_0$ and $\eta$ differ. 
Accordingly the character $\eta$ is $\si$-invariant. 
\end{proof}

The previous statement ensures the assumption \ref{thm_loc_gen_reg}\ref{thm_loc_geni4}. In the proof of Theorem \ref{thm_sec6} we can focus on the case where $d$ is a non-regular number of $(\bG,F)$. 

\begin{proof}[Proof of Theorem \ref{thm_sec6}]
If $d$ is a regular number of $(\bG,F)$ then by the previous results we can apply Theorem \ref{thm_loc_gen_reg} and thus obtain Theorem \ref{thm_sec6} in that case.

In the remaining case we apply Theorem \ref{thm_loc_gen} by deducing the required assumptions from results known in the regular case. According to results from \cite{S10b} the  groups considered in the non-regular case are closely related to groups occuring in the regular case.

Let $d_0:=d$ for odd $d$ and $d_0:=\frac d 2$ otherwise. Let $l'$ be the maximal multiple of $d_0$ with $l'\leq l$. For $a:= \frac {l'}{d_0}$ let $v:=(v_1 \cdots v_{l'-1} \n_{2e_{l'}}(-1))^{a}$. Then $\bT$ contains a Sylow $d$-torus $\bS$ of $(\bG,vF)$ according to \cite[3.2]{S10b}. (This ensures assumption \ref{thm_loc_gen}\ref{5_3_i}.)

Let $C:=\cent {\bG^{vF}}\bS$ and $N:=\norm {\bG^{vF}} \bS$. 
Let $\bG_1$, $\bT_1$ and $\bG_2$ be defined as in \cite[2.2, 2.3]{S10b}. Let $T_1:=\bT_1^{vF}$ and $G_2:=\bG_2^{vF}$. By the proof of \cite[5.2]{S10b} we see $L=T_1\times G_2$, where $d$ is a regular number of $(\bG_1,F|_{\bG_1})$ and $\bS$ is a Sylow $d$-torus of $(\bG_1,F|_{\bG_1})$. 

According to the proof of \cite[5.1]{S10b} the group $N:=\norm {\bG^{vF}} \bS$ satisfies $N=N_1\times G_2$, where $N_1:=\norm {\bG_1} \bS^{vF}$. Accordingly any character $\xi\in\Irr(C)$ can be written
as $\zeta \times \theta$ with $\zeta\in\Irr(T_1)$ and $\theta\in\Irr(G_2)$.
By the proof of \cite[5.2]{S10b} we have $[N_1,G_2]=1$, 
$N_\xi=N_{1,\zeta}\times G_2$ and hence $W_\xi=W'_{\zeta}$ for $W':=N_1/T_1$.

By the description of $\Z(\bG)$ given in \cite[Table 1.12.1]{GLS3} we see that the diagonal automorphism of $\bG^{vF}$ induced by an element of $\bT$ induces
a diagonal automorphism of $\bG_1^{vF}$ on $\bG_1^{vF}$
and a diagonal automorphism of $\bG_2^{vF}$ on $\bG_2^{vF}$. 
Any diagonal automorphism of $\bG^{vF}$ induced by some element of $\bT$ acts trivially on $\Irr(T_1)$. 
Analogously the field automorphism $F_0$ induces field automorphisms of $\bG_1^{vF}$ and of $\bG_2^{vF}$. 

By Theorem \ref{MainTh}, $\theta$
satisfies $(\w N\rtimes E)_{\theta}=\w N_{\theta}\rtimes E_\theta$.
Since $\wN$ and $N$ act the same on $\Irr(T_1)$, this implies for $\xi=\zeta\times \theta$ the equation
$ \wN_{\Ind^N_C(\xi)} \rtimes E_{\Ind^N_C(\xi)} =(\wN E)_{\Ind^N_C(\xi)}$.
This proves \ref{thm_loc_gen}\ref{thm_loc_genii1} for $\calT=\Irr(C)$.

By the proof of Theorem \ref{thm_loc_gen_reg} there exists some $N_1 \rtimes E$-equivariant extension map $\Lambda_1$ with respect to $T_1\lhd N_1$. 
Let $\Lambda$ be the extension map with respect to $C\lhd N$ given by 
$\zeta \times \theta\mapsto \Lambda_1(\zeta)\times \theta$. By this construction $\Lambda$ is $NE$-equivariant as required in \ref{thm_loc_gen}\ref{thm_loc_genii2}. 

It remains to prove that for every $\xi\in\Irr(C)$ the characters of $W_\xi$ satisfy the assumption \ref{thm_loc_gen}\ref{thm_loc_genii3}. The group $W_\xi$ coincides with $W_\zeta$ because of $[N_1,G_2]=1$. 

Let $\w\xi\in\Irr(\w C_\xi\mid \xi)$. If $\w C_\xi=C\cent {\w C} G$ we have $W_{\w\xi}=W_\zeta=W_\xi$. In that case the assumption is clear.
If $\w C=\w C_\xi$ the character $\theta$ is $\w C$-invariant and $W_{\w\zeta}=W_{\w\xi}$ where $\w\zeta$ is an extension of $\zeta$ to $\w T_1$ where $\w T_1:=\{ t\in \bT_1 \mid F(t) t^{-1} \in \bZ(\bG_1)\}$. The groups $W'$, $W_\zeta$ and $W_{\w\zeta}$ occur in the regular situation for the group $(\bG_1,vF|_{\bG_1})$. 
In the previous section we have proven that those groups satisfy \ref{thm_loc_gen_reg}\ref{thm_loc_geni4}. This implies the analogous condition \ref{thm_loc_gen}\ref{thm_loc_genii3}. Altogether this verifies the required assumptions from \ref{thm_loc_gen} and implies the statement of Theorem \ref{thm_sec6}.
\end{proof}

\section{Proof of Theorem \ref{thmA}}\label{end}
For the proof of Theorem \ref{thmA} by an application of Theorem 2.12 of \cite{S12} we construct another extension map and another character correspondence. Now $\bG$ and $\w\bG$ are as in Sect. 3 and 5.

\begin{thm}\label{thm_ext_map_wN}
There exists an $\wN E$-equivariant extension map $\w\Lambda$ for $\Irr(\w C)$ with respect to $\w C\lhd \w N$ that 
is compatible with the action of $\Irr(\w C| 1_C)$ by multiplication.
\end{thm}
\begin{proof}
Let $\Lambda$ be the $NE$-equivariant extension map for $\Irr(C)$ with respect to $C\lhd N$.
In the regular case the existence of $\Lambda$ follows from 
Theorem \ref{vgood} by the proof of Theorem \ref{thm_loc_gen}. 
In the other case such a map was constructed in the proof of Theorem \ref{thm_sec6}.

For characters $\xi:=\zeta \times \theta\in\Irr(C)$ and $\w\xi\in\Irr(\w C\mid \xi)$ with $\Res^{\w C}_C\w\xi=\xi$ we define $\w\Lambda(\w\xi)$ as the common extension of $\w\xi$ and $\Res_{N_{\w\xi}}^{N_{\zeta \times \theta}}(\Lambda(\zeta \times \theta))$, see \cite[4.2]{S09}. 

Let $\calT_2\subseteq \Irr(G_2)$ be a $\w CE$-transversal of the set of non $\w C$-invariant characters of $G_2$. Then $\calT_2$ is a $\w NE$-transversal of the same set. Let $\theta\in\calT_2$ and $\w\xi\in\Irr(\w C\mid \theta)$. Then $\Res^{\w C}_C\w\xi$ is not irreducible. Let $\zeta\in\Irr(T_1)$ with $\w\xi\in\Irr(\w C\mid \zeta \times \theta)$. Then there exists a unique extension of $\w\xi$ in 
$\Irr(\w N_{\w\xi}\mid \Res_{N_{1,\w\xi}}^{N_{1,\zeta}} (\Lambda_1(\zeta)) \times \theta)$. We define this as $\w \Lambda(\w\xi)$. According to Theorem \ref{MainTh} there is a unique way to extend this to a $\w N E$-equivariant extension map $\w\Lambda$. By this construction $\w\Lambda$ has all the required properties. 
\end{proof}

We apply this to prove the following theorem that ensures assumption 2.12(iv) of \cite{S12}.
\begin{thm}\label{thm_bij}
 Let $\ell$ be an odd prime with $\ell\nmid q$ and $d$ the order of $q$ in $(\ZZ/\ell \ZZ)^\times$, $\bS$ a Sylow $d$-torus of $(\bG,F)$, $N\deq \NNN_G(\bS)$ and
 $\w N\deq \NNN_\wG(\bS)$. Let $\cG\deq \cup_{\chi\in\Irrl(G)}\Irr\big (\w G\mid \chi\big)$
 and $\cN\deq \cup_{\psi\in\Irrl(N)}\Irr\big (\w N\mid \psi\big )$. Then there is a
 $(\wG\rtimes D)_{\bS}$-equivariant bijection
$\w \Omega: \cG \longrightarrow \cN$
 with $\w\Omega(\cG\cap\Irr(\w G\mid \nu))= \cN\cap\Irr(\w N\mid \nu)$ for
 every $\nu\in\Irr(\Z(\w G))$, and
 $\w\Omega(\chi\delta)=\w\Omega(\chi)\Res^\wG_{\w N}(\delta)$
 for every $\delta \in \Irr(\w G\mid 1_G)$ and $\chi\in\cG$.
\end{thm}
\begin{proof}
Using the extension map from Theorem \ref{thm_ext_map_wN} the statement follows from the proof given in Section 6 of \cite{CS15}, where it was applied in type $\tA$.
\end{proof}

\begin{proof}[Proof of Theorem \ref{thmA}]
Let $\ell$ be a prime, $S$ a simple group of Lie type $\tC_l$. By \cite[1.1]{S12} and \cite{MS15} we may assume that $\ell$ is different from the defining characteristic and odd. According to \cite[5.2]{CS13} we may assume that the defining characteristic is odd. According to \cite[Tab.~6.1.3]{GLS3} the Schur multiplier of $S$ is non-exceptional and some $G$ defined as in \ref{not_loc_gen} (together with \ref{n_typC}) is the universal covering group of $S$. 

We now review the assumptions of \cite[2.12]{S12}. Arguing as in the proof of Proposition 7.2 of \cite{CS15} we see that assumption 2.12(i) is satisfied with the groups $\w G \rtimes D$ defined as in \ref{not_loc_gen}. Let $d$ be the order of $q$ in $(\ZZ/\ell \ZZ)^\times$ and $\bS_0$ a Sylow $d$-torus of $(\bG,F)$. According to \cite[5.14]{MaH0} the group $N:=\NNN_G(\bS_0)$ is $(\w G\rtimes D)_Q$-stable for some Sylow $\ell$-subgroup $Q$ of $N$ and hence satisfies 2.12(ii). For $\w N:=\NNN_{\wG}(\bS)$ the quotients $\w G/G$ and $\w N/N$ are cyclic and hence 2.12(iii) naturally holds. The remaining assumptions of \cite[2.12]{S12} are the Theorems \ref{MainTh}, \ref{thm_sec6} and \ref{thm_bij}. This finishes our proof. 
\end{proof}

\end{document}